\documentclass[draft,12pt]{amsart}%

\usepackage{amscd,amsmath,amsthm,amssymb}
\usepackage[left]{lineno}
\usepackage{color}
\usepackage{stmaryrd}
\usepackage[utf8]{inputenc}
\usepackage{cleveref}
\usepackage{epstopdf}
\usepackage{graphicx}
\usepackage{xcolor}

\usepackage{tikz}

\definecolor{verylight}{gray}{0.97}
\definecolor{light}{gray}{0.9}
\definecolor{medium}{gray}{0.85}
\definecolor{dark}{gray}{0.6}

\def\NZQ{\mathbb}               
\def\NN{{\NZQ N}}

\def\ZZ{{\NZQ Z}}

\def\KK{{\NZQ K}}

\def\frk{\mathfrak}               

\def\mm{{\frk m}}

\def\KK{{\NZQ K}}

\renewcommand{\qedsymbol}{$\square$} 

\def\G{{\mathcal G}}

\def\P{{\mathcal P}}

\def\dstab{\operatorname{dstab}} 
\def\0b{{\mathbf 0}}

\def\reg{\operatorname {reg}}
\def\sign{\operatorname {sign}}
\def\height{\operatorname{ht}}
\def\depth{\operatorname{depth}}
\def\opn#1#2{\def#1{\operatorname{#2}}} 
\opn\chara{char} \opn\length{\ell} \opn\pd{pd} \opn\rk{rk}
\opn\projdim{proj\,dim} \opn\injdim{inj\,dim} \opn\rank{rank}
\opn\depth{depth} \opn\grade{grade} \opn\height{height}
\opn\embdim{emb\,dim} \opn\codim{codim}

\opn\Tr{Tr} \opn\bigrank{big\,rank}
\opn\superheight{superheight}\opn\lcm{lcm}
\opn\trdeg{tr\,deg}
	\opn\reg{reg} \opn\lreg{lreg} \opn\ini{in} \opn\lpd{lpd}
	\opn\size{size} \opn\sdepth{sdepth}
	\opn\link{link}\opn\fdepth{fdepth}\opn\lex{lex}
	\opn\tr{tr}
	\opn\type{type}
	\opn\gap{gap}
	\opn\arithdeg{arith-deg}
	\opn\HS{HS}
	\opn\GL{GL}
	\opn\div{div} \opn\Div{Div} \opn\cl{cl} \opn\Cl{Cl}
	\opn\Spec{Spec} \opn\Supp{Supp} \opn\supp{supp} \opn\Sing{Sing}
	\opn\Ass{Ass} \opn\Min{Min}\opn\Mon{Mon}
	\opn\Ann{Ann} \opn\Rad{Rad} \opn\Soc{Soc}\opn\Deg{Deg}
	\opn\Im{Im} \opn\Ker{Ker} \opn\Coker{Coker} \opn\Am{Am}
	\opn\Hom{Hom} \opn\Tor{Tor} \opn\Ext{Ext} \opn\End{End}
	\opn\Aut{Aut} \opn\id{id}
	
	\opn\nat{nat}
	\opn\pff{pf}
	\opn\Pf{Pf} \opn\GL{GL} \opn\SL{SL} \opn\mod{mod} \opn\ord{ord}
	\opn\Gin{Gin} \opn\Hilb{Hilb}\opn\sort{sort}
	\opn\PF{PF}\opn\Ap{Ap}
	\opn\mult{mult}
	\opn\bight{bight}
	\opn\aff{aff}
	\opn\relint{relint} \opn\st{st}
	\opn\lk{lk} \opn\cn{cn} \opn\core{core} \opn\vol{vol}  \opn\inp{inp} \opn\nilpot{nilpot}
	\opn\link{link} \opn\star{star}\opn\lex{lex}\opn\set{set}
	\opn\width{wd}
	\opn\Fr{F}
	\opn\QF{QF}
	\opn\G{G}
	\opn\type{type}\opn\res{res}
	\opn\conv{conv}
	\opn\Ind{Ind}
	\opn\gr{gr}
	
	\def\pot#1#2{#1[\kern-0.28ex[#2]\kern-0.28ex]}

	\opn\dirlim{\underrightarrow{\lim}}
	\opn\inivlim{\underleftarrow{\lim}}
	\let\to=\rightarrow
	
	\def\Implies{\ifmmode\Longrightarrow \else
		\unskip${}\Longrightarrow{}$\ignorespaces\fi}
	\def\implies{\ifmmode\Rightarrow \else
		\unskip${}\Rightarrow{}$\ignorespaces\fi}
	\def\iff{\ifmmode\Longleftrightarrow \else
		\unskip${}\Longleftrightarrow{}$\ignorespaces\fi}

	\let\:=\colon
	\newtheorem{Theorem}{Theorem}[section]
	\newtheorem{Lemma}[Theorem]{Lemma}
	\newtheorem{Corollary}[Theorem]{Corollary}
	\newtheorem{Proposition}[Theorem]{Proposition}

	\newtheorem{Example}[Theorem]{Example}
	
	\newtheorem{Definition}[Theorem]{Definition}

	\let\epsilon\varepsilon
	\let\kappa=\varkappa
	\def\qed{\ifhmode\textqed\fi
		\ifmmode\ifinner\quad\qedsymbol\else\dispqed\fi\fi}
	\def\textqed{\unskip\nobreak\penalty50
		\hskip2em\hbox{}\nobreak\hfil\qedsymbol
		\parfillskip=0pt \finalhyphendemerits=0}
	\def\dispqed{\rlap{\qquad\qedsymbol}}
	
	\opn\dis{dis}
	\def\pnt{{\raise0.5mm\hbox{\large\bf.}}}
	
	\opn\Lex{Lex}
	
\begin{document}
	\title {Depth and regularity  of powers of edge ideals of  edge-weighted  trees}
		
\author[T.T. Hien]{Truong Thi Hien}
\address{Faculty of Natural Sciences, Hong Duc University, No. 565 Quang Trung, Hac Thanh, Thanh Hoa, Vietnam}
\email{hientruong86@gmail.com}		
		
\author{Jiaxin Li}
\address{School of Mathematical Sciences, Soochow University, Suzhou, Jiangsu, 215006, P. R.~China}
\email{lijiaxinworking@163.com}

\author{Tran Nam Trung}
\address{Institute of Mathematics, Vietnam Academy of Science and Technology, 18 Hoang Quoc Viet, 10072 Hanoi, Vietnam}
\email{tntrung@math.ac.vn}

\author{Guangjun Zhu$^{\ast}$}
\address{School of Mathematical Sciences, Soochow University, Suzhou, Jiangsu, 215006, P. R.~China}
\email{zhuguangjun@suda.edu.cn}

		 \thanks{ * Corresponding author.}
		
		\thanks{2020 {\em Mathematics Subject Classification}.
			Primary 13C15,13C10; Secondary  05E40, 05C05}

		\thanks{Keywords: Depth, regularity, edge ideal, strictly increasing weighted tree}

		
		

		\begin{abstract}  For an increasing weighted tree $G_\omega$, we  obtain an asymptotic value and  a sharp bound on the index stability of the depth function of its  edge ideal $I(G_\omega)$.  Moreover, if $G_\omega$ is a strictly increasing weighted tree, we provide the minimal free resolution of $I(G_\omega)$ and an  exact formula for the regularity of all powers of   $I(G_\omega)$.
		\end{abstract}
		\setcounter{tocdepth}{1}
		
		\maketitle
\section{Introduction}
Let $G$ be a finite simple  graph with the vertex set $V(G)$ and  the edge set $E(G)$. We write $xy$ for  an  edge of $G$ with  endpoints $x$ and $y$.
Suppose   $\omega \colon E(G) \to \ZZ_{>0}$ is a weight function on $E(G)$. The pair $(G,\omega)$  is called  an {\it edge-weighted graph} (or simply a weighted graph) with the underlying graph  $G$ and is denoted by  $G_\omega$.   

Let $S=\KK[x_{1},\dots, x_{n}]$ be a polynomial ring of $n$ variables over a field $\KK$. Assume that  $V(G)=\{x_1,\ldots,x_n\}$.  The  {\em edge-weighted ideal}  (or simply the edge ideal)  of   $G_\omega$  was introduced in \cite{PS} and  is defined as the ideal of $S$ by
\[
I(G_\omega) = ((x_ix_j)^{\omega(x_ix_j)}\mid x_i x_j\in E).
\]
If $\omega \colon E(G) \to \ZZ_{>0}$ is a constant function with value one, then $I(G_\omega)$
is just the edge ideal $I(G)$ of the underlying graph $G$. This ideal has been  studied extensively in the literature (see \cite{ABS, BHT,FM,LTT,M,T}).

Brodmann  proved in \cite{B} that for a  proper homogeneous ideal   $I\subset S$, the depth function $\depth (S/I^t)$  is constant for $t\gg 0$ and bounded above by $n-\ell (I)$, where $\ell(I)$ is the analytic spread of $I$. If the associated graded ring of $I$ is Cohen–Macaulay, then   $\depth (S/I^t)$ attains this upper bound (see \cite[Proposition 3.3]{EH}).  For example, if $I$ is a polymatroidal ideal, then this occurs
(see \cite[Corollary 3.5]{HRV}).
 The first position from which  $\depth (S/I^t)$ becomes constant is called the stability index of the depth function of $I$ and is denoted by $\dstab(I)$. Brodmann's result raises the question of  whether  $\dstab(I)$
 and the asymptotic value of  $\depth (S/I^t)$ can be characterized explicitly in terms of  $I$, but this  problem is difficult  because the depth function of an ideal can be any convergent function (see \cite{HHTT, HH1}).

 There  are few basic results about the nature of $\dstab(I)$ and the asymptotic value of  $\depth (S/I^t)$, even when $I$ is a squarefree monomial ideal.  However, if  $I$ is the edge ideal of a simple graph, then these problems are completely solved (see \cite{LTT, T}). Note that the associated graded ring of such an ideal may not be  Cohen-Macaulay.

For a  proper homogeneous ideal  $I\subset S$, it is well known that  the regularity  function $\reg(I^t)$ is asymptotically a linear function for $t\gg  0$ (see \cite{CHT,K}). That is, there exist  constants  $a$, $b$ and $t_0$ such that $\reg(I^t)=at+b$ for all $t\geq t_0$. While $a$ is well-defined (see \cite[Theorem 5]{K}), a little information is known about $b$ and $t_0$. Therefore, it is of interest to determine the exact form of this linear function and the stability index  $t_0$ at which $\reg(I^t)$  becomes linear (see \cite{Con, EH1, R}). It turns out that, even in the case of square-free monomial ideals, finding the linear function and $t_0$ is challenging (see \cite{ABS, BHT,JS1, MSY}).

 This paper is concerned with the regularity and depth of powers of the edge ideal of a weighted graph $G_\omega$. When $G_\omega$ is a weighted star graph, or an integrally closed weighted   path, or an integrally closed  weighted  cycle, the second and fourth authors, along with other collaborators, provided exact formulas for the regularity and depth of  powers of the edge ideal $I(G_\omega)$ (see \cite{LVZ,ZCLY,ZDCL,ZLCY}).  If $G_\omega$ is an integrally closed weighted  tree, the second and fourth authors provided exact formulas for the regularity of the edge ideal and  offered some linear upper bounds on the regularity of its powers (see \cite{LZD}).
 
 We say that $G_\omega$ is an increasing weighted tree if $G$  is a tree and there exists a vertex $r$, which is called a root of $G_\omega$, such that the weight function on every simple path from a leaf to the root $r$ is increasing. That is,
 if
$$\P\colon v_1\to v_2\to v_3\to \cdots \to v_k =r$$
is a  simple path from a leaf $v_1$ to $r$ of length at least $2$, then $\omega(v_iv_{i+1}) \leqslant \omega(v_{i+1}v_{i+2})$ for $i=1,\ldots, k-2$.
 If this  path also satisfies the condition that  $\omega(v_1v_2)=\omega(v_2v_3)$, then $v_2$ is called a {\it special} vertex. Let $s(r,G_\omega)$ be the number of special vertices and 
$$s(G_\omega) = \min\{s(r,G_\omega)\mid r \text{ is a root of } G_\omega\}.$$
 
If every simple path $\P \colon v_1\to v_2\to v_3\to \cdots \to v_k =r$
from a leaf $v_1$ to $r$ of length at least $2$ satisfies  $\omega(v_iv_{i+1}) < \omega(v_{i+1}v_{i+2})$ for $i=1,\ldots, k-2$, then $G_\omega$ is said to be 
 a strictly increasing weighted tree.

 Our first main result is the following theorem.

\medskip

\noindent {\bf Theorem \ref{theorem-depth}. \it 
Let $G_\omega$ be an increasing weighted tree. Then, 
$$\depth(S/I(G_\omega)^t)=1 \text{ for all } t\geqslant s(G_\omega)+1.$$
In particular, $\dstab(I(G_\omega))\leqslant s(G_\omega)+1$.
}
\medskip

We next prove that if $G_\omega$ is a strictly increasing weighted tree, then the Taylor complex of the edge ideal  $I(G_\omega)$ is the minimal free resolution of  $S/I(G_\omega)$ (see Theorem \ref{minimal}). As a result, we derive a formula for $\reg (I(G_\omega))$. This formula plays a key role in the proof of the following theorem.

\medskip

\noindent {\bf Theorem \ref{reg-power}. \it Let $G_\omega$ be a strictly increasing weighted tree. Then
$$\reg(I(G_\omega)^t)=2d(t-1) + \reg (I(G_\omega)) \text{ for all } t\geqslant 1,$$
where $d=\max\{\omega(e)\mid e\in E(G)\}$.
}
\medskip

The paper is organized as follows:  Section \ref{sec:prelim}  introduces the notation and provides background information. Section \ref{sec:depth} proves Theorem  \ref{theorem-depth}  and characterizes an increasing weighted tree, $G_\omega$, such that  the depth function, $\depth (S/I(G_\omega)^t)$, is constant. Section \ref{sec:regularity}  proves that the Taylor complex of  $I(G_\omega)$, where $G_\omega$ is a strictly increasing weighted tree, is the minimal free resolution of $S/I(G_\omega)$. This section also proves Theorem \ref{reg-power}.

\section{Preliminaries}
 \label{sec:prelim}

In this section, we collect definitions and basic facts that will be used throughout this paper. For more details, the reader is referred to \cite{D, HH2}.

Let $\KK$ be a field, and let $S=\KK[x_1, \ldots,x_n]$ be the polynomial ring in  $n$ variables $x_1,\ldots,x_n$ over the field $\KK$. Let $\mm=(x_1,\ldots,x_n)$ be the maximal homogeneous ideal of $S$. The focus of our work is the depth and Castelnuovo-Mumford regularity of homogeneous ideals of $S$, which can be defined in various ways.  For our purposes, we will use the definition that employs the minimal free resolution of graded $S$-modules.
Let $M$ be a finitely generated graded $S$-module, and let
$$0\to \bigoplus_{j\in\ZZ} S(-j)^{\beta_{p,j}(M)}\to\cdots\to\bigoplus_{j\in\ZZ} S(-j)^{\beta_{1,j}(M)}\to \bigoplus_{j\in\ZZ} S(-j)^{\beta_{0,j}(M)}\to M\to 0$$
be its minimal graded free resolution. Then, the 
projective dimension of $M$ is 
$$\pd(M)=p.$$ 
The regularity of $M$ is defined by
$$\reg(M )= \max\{j-i \mid \beta_{i,j}(M)\ne 0\}.$$
The depth of $M$ is given by the Auslander-Buchsbaum formula
$$\depth (M)= n-\pd(M) = n-p.$$
The number $\beta_{i,j}(M)$ is called the $(i,j)$-th graded Betti number of $M$, and the number
$$\beta_i(M) =  \sum_{j\in\ZZ} \beta_{i,j}(M)$$
is called the $i$-th Betti number of $M$.

The support of a monomial $f$ is $\supp(f) = \{x_i\mid x_i\text{ divides } f\}$. The support of a finite set $W$ of monomials is $\supp(W)=\{x_i\mid x_i\in \supp(f) \text{ for some } f\in W\}$.

For a monomial ideal $I\subseteq S$, let  $\mathcal{G}(I)$ denote the unique minimal set of  monomial generators of $I$, and the support of $I$, denoted by $\supp(I)$, is $\supp(\mathcal{G}(I))$.

\begin{Lemma} {\em (\cite[Lemma 2.2 and Lemma 3.2]{HT})}\label{sum1} 
	Let $S_1 = \mathbb{K}[x_1, \ldots, x_m]$, $S_2 = \mathbb{K}[x_{m+1}, \ldots, x_n]$ and $S = \mathbb{K}[x_1, \ldots, x_n]$ be three polynomial rings over $\mathbb{K}$, $I \subseteq S_1$ and $J \subseteq S_2$ be two proper non-zero homogeneous ideals. Then we have  
	\begin{itemize}		
	\item[(1)] $\depth(S/(I+J)) = \depth (S_1/I)+ \depth (S_2/J).$
        \item[(2)] $\reg(S/(I+J)) = \reg (S_1/I)+ \reg (S_2/J).$
	\end{itemize}
\end{Lemma}

\begin{Lemma}{\em (\cite[Lemma 3.1]{HT})}\label{exact}
Let $0 \longrightarrow M \longrightarrow N \longrightarrow P \longrightarrow 0$ be a short exact sequence of finitely generated graded $S$-modules. Then 
$$\reg(N) \le \max\{\reg(M), \reg(P)\}.$$ 
The equality holds if  $\reg(P) \ne \reg(M)-1$.
\end{Lemma}
		
		\begin{Lemma}{\em (\cite[Theorem 7.1]{H})}\label{depth}
			Let $I\subseteq S$ be a monomial ideal. Then
			\[
			\depth(S/I) = \min\{ \depth(S/\sqrt{I:f})\mid f \text{ is a monomial such that } f \notin I\}.
			\]
		\end{Lemma}
\medskip

A simple graph $G=(V(G),E(G))$ is a graph  without loops or multiple edges, where $V(G)$ and  $E(G)$  are the sets of   vertices  and  edges of $G$, respectively.
 A graph $H$ is called an induced subgraph of $G$ if $V(H)\subset V(G)$, and  for any two vertices $u,v\in V(H)$, $uv\in E(H)$ if and only if $uv\in E(G)$. The induced subgraph of $G$ on a subset $W\subseteq V(G)$ is obtained from $G$  by  deleting  the vertices not in $W$ and their incident edges.
 We also denote the induced subgraph of $G$ on the set $V(G)\setminus W$ by $G\setminus W$, and if $W =\{v\}$, then  $G\setminus v$ stands for $G\setminus \{v\}$.
 
 Given a graph $G$ and a vertex  $v$, let $N_G(v)=\{u \in V(G) \mid uv \in E(G)\}$ be the neighborhood of $v$. The degree of a vertex $v$, denoted by $\deg_G(v)$, is the cardinality of its neighborhood, i.e. $\deg_G(v)=|N_G(v)|$. A vertex $v$ is a leaf if $\deg_G(v)=1$, meaning it has a unique neighbor. An edge containing a leaf is called a pendant edge. In this paper, we denote $L_G(v)$ as the set of leaves in $G$  that are adjacent to $v$.
 
A path in $G$ is a sequence of vertices $v_1,v_2,\ldots, v_k$ such that $v_iv_{i+1}\in E(G)$ for $i=1,\ldots,k-1$. If all  the  vertices of a path are distinct, then it is called a simple path. The length of a path is the number of edges in the  path. We write
$$v_1\to v_2\to\cdots \to v_k$$
to indicate a path from $v_1$ to $v_k$. 
A cycle of length $k$, where $k\geqslant 3$, is a path $v_1,v_2,\ldots, v_k,v_1$ where $v_1,\ldots,v_k$ are distinct. A tree is a connected simple graph without cycles. A graph $G$ is bipartite if $V(G)$ admits a partition into  two subsets, $X$ and $Y$, such that every edge has one vertex in $X$ and  one in $Y$. In this case, the pair  $(X,Y)$ is called a bipartition of $G$. If every  vertex in $X$ is adjacent to every vertex in $Y$, then $G$ is called a complete bipartite graph and is denoted by $K_{X,Y}$. If $X =\{x\}$, then $K_{X,Y}$ is called a star graph  with center $x$. It is well known  that a graph is bipartite if and only if it has no odd cycles. (See, for example, \cite[Proposition 1.6.1]{D}).  In particular, a tree is a bipartite graph.

Let $\omega\colon E(G)\to \ZZ_{>0}$ be a weight function on  $E(G)$ and let $H$ be a subgraph of $G$.  Then, we can restrict  the function $\omega$ to  $E(H)$ to obtain the weighted graph $H_\omega$. This means that the weight of an edge $uv$ in $H$ is simply $\omega(uv)$.

\begin{Lemma}{\em (\cite[Lemma 2.1]{LTZ2})}\label{colon}
	Let $I$ be a monomial ideal and let $x^py^q$ be a monomial in $\mathcal{G}(I)$, where $p \geqslant 1$ and $q\geqslant 1$, and $x$ and $y$ are variables. For any   $f$ in $\mathcal{G}(I)$, $f$ satisfies 
	\begin{enumerate}
		\item if  $f\ne x^py^q$, then $y\nmid f$,
		\item if  $x\mid f$, then $\deg_x(f)\geqslant p$.
	\end{enumerate}
	Then $(I^t \colon x^py^q)= I^{t-1}$ for all $t\geqslant 2$.
\end{Lemma}

Recall that $G_\omega$ is an increasing weighted tree if $G$  is a tree and if there exists a vertex $r$ such that the weight function on every simple path from a leaf to $r$ is increasing. In other words,
if
$$v_1\to v_2\to v_3\to \cdots \to v_k =r$$
is a  simple path from a leaf $v_1$ to $r$ of length at least $2$, then $\omega(v_iv_{i+1}) \leqslant \omega(v_{i+1}v_{i+2})$ for $i=1,\ldots, k-2$. In this case, $r$ is called a root of $G_\omega$,
and $(G_\omega,r)$ is an increasing weighted tree. Obviously, $G_\omega$ is an increasing weighted tree if $G$ is a star graph.

\begin{Lemma}{\em (\cite[Lemma 1.5]{LTZ2})}\label{path}
 Assume that $(G_\omega,r)$ is an increasing weighted tree. If $G$ is not a star  graph centered at $r$, then there is a longest path 
$$r= v_0\to v_1\to\cdots\to v_k$$
in $G$ from  $r$ such that
\begin{enumerate}
    \item $v_k$ is a leaf;
    \item if $u\in N_G(v_{k-2})$ is a non-leaf, then $\omega(v_{k-1}v_{k-2})\leqslant \omega(v_{k-2}u)$;
    \item $N_G(v_{k-1})$ has only one non-leaf $v_{k-2}$;
    \item $\omega(v_{k-1}u) \leqslant \omega(v_{k-1}v_{k-2})$ for all $u\in N_G(v_{k-1})$;
    \item $\omega(v_{k-1}v_k) \leqslant \omega(v_{k-1}u)$ for all $u\in N_G(v_{k-1})$.
\end{enumerate}
\end{Lemma}

\begin{Lemma}\label{depth-decreasing}
 If $G_\omega$ is an increasing weighted tree, then $$\depth(S/I(G_\omega)^{t+1})\leqslant \depth(S/I(G_\omega)^{t}) \ \text{ for all } t\geqslant 1.$$
\end{Lemma}
\begin{proof} Let $r$ be the root of $G_\omega$ and let $I=I(G_\omega)$. If $G$ is a star graph  centered at $r$, then $\depth(S/I^t)=1$ for all $t\geqslant 1$, according to  \cite[Theorems 3.1 and 3.3]{ZDCL}. If $G$ is not a star graph centered at $r$, then by Lemma \ref{path}, there is a pendant  edge $xy$ of $G$ such that $\deg_G(y)=1$ and $\omega(xy)\leqslant \omega(xv)$ for every $v\in N_G(x)$. By Lemma \ref{colon}, $(I^{t+1}\colon (xy)^{\omega(xy)}) = I^{t}$. Together with Lemma \ref{depth}, this yields
$$\depth(S/I^{t+1}) \leqslant \depth(S/(I^{t+1}\colon (xy)^{\omega(xy)}) )=\depth(S/I^t),$$
as required.
\end{proof}

\begin{Lemma}{\em (\cite[Lemma 2.2]{LTZ2})} \label{max-ideal} 
If $G_\omega$ is an increasing weighted tree, then $\mm\notin \operatorname{Ass}(I(G_\omega)^t)$ for all $t \geq 1$.
\end{Lemma}

\begin{Definition} Let $(G_\omega,r)$ be an increasing weighted tree. 
An edge $uv$ of $G$ is special if there is a simple path in $G$ to $r$ in the form
$$v_1\to u=v_2\to v=v_3\to\cdots \to v_k=r,$$
 where $k\geqslant 3$ and $\omega(uv_1)=\omega(vu)$. Let  $S(r,G_\omega)$ be the set of special edges.
\end{Definition}

\begin{Lemma}\label{sG-Equality} If $(G_\omega,r)$ be an increasing weighted tree, then $s(r,G_\omega) = |S(r,G_\omega)|$.
\end{Lemma}
\begin{proof} If $G$ is a star graph with center $r$, then  $s(r,G_\omega)=|S(r,G_\omega)|=0$. Now, we assume that $G$ is not a star with center $r$. Let
$ v_0\to v_1\to v_2\to\cdots\to v_{k-1}\to v_k=r$ be any simple path  in $G$ to $r$ with $k\geqslant 2$. For each edge $v_iv_{i+1}$ of $G$,  we denote $(v_i,v_{i+1})$ to represent a directed edge with $v_i$ as the starting vertex and $v_{i+1}$ as the ending vertex. It is easy to see that the edge  $v_1v_2$ in $G$ with direction $(v_1,v_2)$ is a special edge of $(G_\omega,r)$ if and only if $v_1$ is a special vertex of $(G_\omega,r)$. It follows that $s(r,G_\omega) = |S(r,G_\omega)|$.
\end{proof}

\section{Depth of powers of the edge ideal of weighted trees} 
\label{sec:depth}

In this section, we study the asymptotic behavior of the depth function of the edge ideal of an increasing weighted tree. We always assume that $G$ is a tree with bipartition $(U,V)$ and $V(G)=\{x_1,\ldots,x_n\}$; and that $(G_\omega,r)$ is an increasing weighted tree. Let $K_{U,V}$ be the  complete bipartite graph with  bipartition $(U,V)$. For a positive integer $k$, the notation $[k]$ denotes the set $\{1,2,\dots,k\}$.

For any integer vector $\mathbf{a}= (a_1, \ldots, a_n) \in \mathbb{N}^n$, define the monomial $x^{\mathbf{a}} = \prod\limits_{i=1}^{n} x_i^{a_i}$ and write $\deg(x^{\mathbf{a}})=\sum\limits_{i=1}^{n}a_i$ and $\deg_{x_i}(x^{\mathbf{a}}) = a_i$ for each $i \in [n]$.

\begin{Definition} \label{notation}
For each $v\in V(G)$, define 
\[
\mu(v) = \max\{\omega(uv)\mid u\in N_G(v)\}.
\]
Let
$$A(G_\omega) = \left\{ v \in V(G) \;\middle|\;  
       	  \begin{aligned} 
       	  	&\text{there a simple path } v=v_1\to v_2\to \cdots\to v_{2m} \\ 
       	  	&\text{with } m\geqslant 1 \text { and }\omega(v_{2m-1}v_{2m})<\mu(v_{2m})
       	  \end{aligned}\right\},          
$$
and
$$f(r,G_\omega)=\left(\prod_{xy\in S(r,G_\omega)}(xy)^{\omega(xy)}\right)\left(\prod_{v\in V(G)}v^{\mu(v)-1}\right).$$
\end{Definition}

\begin{Lemma} \label{vInside} For any $v\in A(G_\omega)$, let  $v=v_1\to v_2\to \cdots \to v_{2m-1}\to v_{2m}$ 
be the  shortest path in $G$ such that  $\omega(v_{2m-1}v_{2m}) < \mu({v_{2m}})$. Then,  for all $i\in [m-1]$,  $v_{2i}v_{2i+1} \in S(r,G_\omega)$.
\end{Lemma}
\begin{proof} From the assumption we have
\begin{equation}\label{firstV1}
    \omega(v_{2i-1}v_{2i})\geqslant \mu(v_{2i})\geqslant \omega(v_{2i}v_{2i+1}),\ \text{ for any }  i\in [m-1].
\end{equation}
Let $u\in N_G(v_{2m})$ such that $\omega(v_{2m}u) = \mu(v_{2m})$, then $\omega(v_{2m-1}v_{2m}) < \omega(v_{2m}u)$. Let $W = \{v_1,\ldots,v_{2m}\}$. 
We will now consider two cases:
\medskip

{\it Case $1$}: $r\in W$. Then,  $r=v_{2m}$. Indeed, assume that $r\ne v_{2m}$ so that  $r = v_i$ for some $i\in [2m-1]$. Choose the simple path
$$u\to v_{2m}\to v_{2m-1}\to \cdots \to v_i,$$
then we have  $\omega(uv_{2m})\leqslant \omega(v_{2m}v_{2m-1})$, which contradicts the fact that  $\omega(v_{2m-1}v_{2m}) < \omega(v_{2m}u)$, so $r=v_{2m}$. Now, consider the simple path
$$v=v_1\to v_2\to \cdots \to v_{2m-1}\to v_{2m}=r$$ 
we have $\omega(v_{2i-1}v_{2i})\leqslant \omega(v_{2i}v_{2i+1})$, for any $i\in [m-1]$.  According to Eq. $(\ref{firstV1})$, $\omega(v_{2i-1}v_{2i})=\omega(v_{2i}v_{2i+1})$, which 
implies that $v_{2i}v_{2i+1} \in	S(r,G_\omega)$ for  any $i\in [m-1]$.

\medskip

{\it Case $2$}: $r\notin W$. Then, we can take a simple path to $r$ in the form $u_1\to u_2 \to \cdots \to u_k=r$ such that $u_1\in W$ and $u_2,\ldots, u_k \notin W$. In  this case,  $u_1 = v_{2m}$. Indeed, if $u_1\ne v_{2m}$, then  $u_1= v_i$ for some $i\in [2m-1]$. From the simple path
$$u\to v_{2m}\to v_{2m-1}\to \cdots\to  v_{i+1}\to v_i=u_1 \to u_2\to  \cdots \to u_{k-1}\to u_{k}=r,$$ 
we get $\omega(uv_{2m})\leqslant \omega(v_{2m}v_{2m-1})$. This contradicts the fact that  $\omega(v_{2m-1}v_{2m}) < \omega(v_{2m}u)$. 
Now, consider the simple path
$$v=v_{1}\to v_{2}\to \cdots\to  v_{2m-1}\to v_{2m}=u_1 \to u_2\to  \cdots \to u_{k-1}\to u_{k}=r,$$ 
we have $\omega(v_{2i-1}v_{2i})\leqslant \omega(v_{2i}v_{2i+1})$, for any $i\in [m-1]$.  According to Eq. $(\ref{firstV1})$, $\omega(v_{2i-1}v_{2i})=\omega(v_{2i}v_{2i+1})$, which 
implies that $v_{2i}v_{2i+1} \in	S(r,G_\omega)$ for  any $i\in [m-1]$.
\end{proof}

\begin{Lemma} \label{uvOuside} Let $u,v\in V(G)\setminus  A(G_\omega)$. Then, for every simple path from $u$ to $v$ such that $ u=v_1\to v_2\to \cdots \to v_{2m-1}\to v_{2m}=v$ and  $m\geqslant 2$,  we have  $v_{2i}v_{2i+1}\in S(r,G_\omega)$  for all  $i\in[m-1]$.
\end{Lemma}
 \begin{proof} 
By  the hypothesis $u,v \notin A(G_\omega)$, we can conclude that for  any $i\in[m-1]$, 
\begin{equation}\label{firstV}
    \omega(v_{2i-1}v_{2i})\geqslant \mu(v_{2i})\geqslant \omega(v_{2i}v_{2i+1})\text{ and }\omega(v_{2i+1}v_{2i+2})\geqslant\mu(v_{2i+1})\geqslant \omega(v_{2i}v_{2i+1}).
\end{equation}
Let $W = \{v_1,\ldots,v_{2m}\}$. We will now consider two cases.

{\it Case $1$}: If $r\in W$, then, by symmetry, we can assume that $r = v_{2j}$ for some $j\in[m]$. In this case, there are two induced subpaths  to the root $r$ of the form  $v_{1}\to v_{2}\to\cdots \to v_{2j}=r$ and $v_{2m}\to v_{2m-1}\to\cdots \to v_{2j+1}\to v_{2j}=r$. By   \cite[Lemma 1.3]{LTZ2}, these  two simple paths are increasing. Therefore, for these paths, we have $\omega(v_{2i-1}v_{2i})\leqslant \omega(v_{2i}v_{2i+1})$  for any  $i\in[j-1]$, and   $\omega(v_{2i+1}v_{2i+2})\leqslant \omega(v_{2i}v_{2i+1})$  for any  $j\le i\le m-1$, respectively. Combining condition  $(\ref{firstV})$ , we can obtain that 
$v_{2i}v_{2i+1}\in S(r,G_\omega)$  for any $i\in[m-1]$. 

{\it Case $2$}: If $r\notin W$, then we can take a simple path to $r$ in the form
 $u_1\to u_2 \to \cdots \to u_k=r$ such that $u_1\in W$ and $u_2,\ldots, u_k \notin W$. By symmetry, we can assume that $u_1=v_{2j}$ for some $j\in [m]$.
Thus, there are two paths to the root $r$:  $v_1\to v_2\to\cdots\to v_{2j-1}\to  v_{2j}=u_1 \to u_{2}\to \cdots \to v_{k-1}\to u_k=r$ and $v_{2m}\to v_{2m-1}\to\cdots\to v_{2j+1}\to  v_{2j}=u_1 \to u_{2}\to \cdots \to v_{k-1}\to u_k=r$.  Using the same argument as in Case $1$, we can show that $v_{2i}v_{2i+1}\in S(r,G_\omega)$ for any $i\in[m-1]$. 
\end{proof}

\begin{Lemma}\label{L1} $(A(G_\omega))+I(K_{U,V}) \subseteq (I(G_\omega)^t \colon f(r,G_\omega))$, where $t=|S(r,G_\omega)|+1$.
\end{Lemma}
\begin{proof} Let $I = I(G_\omega)$ and $f=f(r,G_\omega)$.  We divide the proof into two steps:

{\it Step $1$}:  $zf\in I^t$ for every $z\in A(G_\omega)$. Indeed, let $z=v_1\to v_2\to \cdots\to v_{2m-1}\to v_{2m}$
be the shortest path in $G$ such that $\omega(v_{2m-1}v_{2m}) < \mu(v_{2m})$. By Lemma \ref{vInside}, we have
$$v_{2i}v_{2i+1}\in S(r,G_\omega),\ \text{ for any } i\in[m-1].$$

Since $t-1= |S(r, G_\omega)|$, we can write $f$ as
\begin{equation}\label{EQ01}
f = f_1\cdots f_{t-m}\left(\prod_{i=1}^{m-1} (v_{2i}v_{2i+1})^{\omega(v_{2i}v_{2i+1})}\right)\left(\prod_{v\in V(G)}v^{\mu(v)-1}\right),
\end{equation}
where $f_1\cdots f_{t-m}\in \mathcal G(I)$.  On the other hand, let $W=V(G)\setminus\{v_1,\ldots,v_{2m}\}$, then 
\begin{equation*}
    z\left(\prod_{i=1}^{m-1} (v_{2i}v_{2i+1})^{\omega(v_{2i}v_{2i+1})}\right)\left(\prod_{v\in V(G)}v^{\mu(v)-1}\right)=g_1g_2\prod_{i=1}^m (v_{2i-1}v_{2i})^{\omega(v_{2i-1}v_{2i})}\in  I^{m},
\end{equation*}
where
$$g_1 = z^{\mu(z)-\omega(zv_2)}\left(\prod_{i=1}^{m-1}v_{2i+1}^{\mu(v_{2i+1})+\omega(v_{2i}v_{2i+1})-\omega(v_{2i+1}v_{2i+2})-1}\right)\left(\prod_{v\in W} v^{\mu(v)-1}\right)
\in S,
$$
and
$$g_2 = v_{2m}^{\mu(v_{2m})-\omega(v_{2m-1}v_{2m})-1}\left(\prod_{i=1}^{m-1}v_{2i}^{\mu(v_{2i})+\omega(v_{2i}v_{2i+1})-\omega(v_{2i-1}v_{2i})-1}\right)\in S.
$$
By the expression  $(\ref{EQ01})$ of $f$, we obtain $zf\in I^t$. 

{\it Step $2$}:  $uvf\in I^t$ for every $u\in U$ and any $v\in V$. Indeed, if $u\in A(G_\omega)$ or $v\in A(G_\omega)$, then the result follows from Case $1$.
In the following, we can assume that $u,v\notin A(G_\omega)$. 
By the choice of $u$ and $v$, there exists a simple path from $u$ to $v$ 
$$u=u_1\to u_2\to \cdots \to u_{2\ell-1}\to u_{2\ell} = v.$$
According to  Lemma \ref{uvOuside}, we have
$$u_{2i}u_{2i+1} \in S(r,G_\omega),\ \text{ for any } i\in[\ell-1].$$
Using similar arguments as in  Case $1$, we can verify that $uvf\in I^t$, and the result follows.
\end{proof}

\begin{Lemma}\label{L2} $(I(G_\omega)^t \colon f(r,G_\omega)) \subseteq (A(G_\omega))+I(K_{U,V})$, where $t=|S(r,G_\omega)|+1$.
\end{Lemma}
\begin{proof} Let $I = I(G_\omega)$ and $f=f(r,G_\omega)$. We will prove the statement by induction on $n=|V(G)|$. 
If $G$ is a star graph with a center $r$, then  $|S(r,G_\omega)|=0$, and
$f = r^{\mu(r)-1}\prod_{v \in V} v^{\omega(rv)-1}$. In this case, $A(G_\omega) = \{v\in V\mid \omega(rv) < \mu(r)\}$,  where  $U=\{r\}$ and $V=N_G(r)$. It is easy to see that 
 $(I\colon f)=(A(G_\omega)) + I(K_{U,V})$, and the result holds.

In the following, we can assume that $G$ is not a star graph. Then $n\ge 4$. Let $g\in \mathcal{G}(I^t:f)$, then $gf\in I^t$.  We
can write $gf$  as
\begin{equation}\label{EQ5}
gf = hf_1f_2\cdots f_t,
\end{equation}
where $h$ is a monomial and each $f_i\in \mathcal{G}(I)$.   We consider the following two cases.

\medskip
{\it Case $1$}: If there is  a pendant edge $xy$ in $G$ that satisfies the following four conditions:  $\deg_G(y)=1$, $y\ne r$, $y\nmid g$  and $S(r,G_\omega)=S(r,G'_\omega)$ where $G'=G\setminus y$.
Then,  $y\nmid f_i$ for  all  $i\in[t]$. Indeed, if $y| f_j$ for some $j\in [t]$, then $f_j = (xy)^{\omega(xy)}$, since  $\deg_G(y)=1$. Thus, by the expression  $(\ref{EQ5})$,  $\deg_y(gf) \geqslant \deg_y(f_j)=\omega(xy)$.  However, since $y\nmid g$, 
$\deg_y(gf) = \deg_y(f)=\mu(y)-1=\omega(xy)-1$,  which is 
a contradiction. This implies that $f_i \in \mathcal G(I(G'_\omega))$ for all $i\in [t]$.

Since $\deg_y(gf)=\omega(xy)-1$ and $y\nmid f_i$ for  all  $i\in[t]$, by the expression  $(\ref{EQ5})$, we have $y^{\omega(xy)-1}\mid h$. Let $h=y^{\omega(xy)-1}h_1$ for some monomial $h_1$.
Thus, the expression  $(\ref{EQ5})$ becomes $gf=y^{\omega(xy)-1}h_1f_1f_2\cdots f_t$.
Since $S(r,G_\omega)=S(r,G'_\omega)$,   $f = f'y^{\omega(xy)-1}$ and  $gf' = h_1f_1\cdots f_t \in I(G'_\omega)^t$, where $f'=f(r,G'_\omega)$. By the induction hypothesis,  $g\in (A(G'_\omega)) + I(K_{U',V'})\subseteq (A(G_\omega)) + I(K_{U,V})$, where $(U',V')$ is a bipartition of $G'$.
\medskip

{\it Case $2$}: Assume that no pendant edge of $G$  satisfies Case $1$ and that  $g\notin (A(G_\omega))$.  By Lemma \ref{path}, there is a longest simple path $\P: r = v_0\to v_1\to\cdots\to v_{s-1}\to v_s$ in $G$ from $r$ such that $s\geqslant 2$ and 
\begin{itemize}
	\item[(i)] $v_s$ is a leaf;
	\item[(ii)] if $u\in N_G(v_{s-2})$ is a non-leaf, then $\omega(v_{s-1}v_{s-2})\leqslant \omega(v_{s-2}u)$;
	\item[(iii)] $N_G(v_{s-1})$ has only one non-leaf $v_{s-2}$;
	\item[(iv)] $\omega(v_{s-1}u) \leqslant \omega(v_{s-2}v_{s-1})$ for all $u\in N_G(v_{s-1})$;
	\item[(v)] $\omega(v_{s-1}v_s) \leqslant \omega(v_{s-1}u)$ for all $u\in N_G(v_{s-1})$.
\end{itemize}
Then, $\omega(v_{s-1}v_s)=\omega(v_{s-2}v_{s-1})$. Indeed, By (v), we can assume that $\omega(v_{s-1}v_s) < \omega(v_{s-2}v_{s-1})$. By choosing $u_1=v_s$  and $u_2=v_{s-1}$ in the definition  of $A(G_\omega)$, 
we can deduce that $v_s\in A(G_\omega)$ and $S(r,G_\omega)=S(r,(G\setminus v_s)_\omega)$. Therefore,  $v_s\nmid g$. Since $s\geqslant 2$, we have $r\notin L_G(v_{s-1})$ and  $v_s\ne r$.
From the simple path,  we know that there is  a pendant edge $v_{s-1}v_s$ in $G$ that satisfies the four conditions that  $\deg_G(v_s)=1$, $v_s\ne r$, $v_s\nmid g$  and $S(r,G_\omega)=S(r,G''_\omega)$ where $G''=G\setminus v_s$.
This contradicts the assumption.
Therefore, $\omega(v_{s-1}v_s)=\omega(v_{s-2}v_{s-1})$. Combining  the conditions (iii), (iv), and (v), we can conclude that  $\omega(v_{s-1}u)=\mu(v_{s-1})$ for every $u\in N_G(v_{s-1})$.

\medskip
Next,  we will show that $g\in (A(G_\omega)) + I(K_{U,V})$ by studying the subgraph $G''=G\setminus v_s$ and  applying  induction. We will consider the following two subcases:

\medskip
{\it Subcase $2.1$}: If  $|L_G(v_{s-1})| \geqslant 2$, then $S(r,G_\omega)=S(r,G''_\omega)$ since $\omega(v_{s-1}u)=\omega(v_{s-2}v_{s-1})$ for every $u\in L_G(v_{s-1})$.  Thus, $v_s| g$ since $v_s$ satisfies the other three conditions in Case $1$. 
When  $v_{s-1}|g$,   $g\in I(K_{U,V})$. Now, we assume that $v_{s-1}\nmid g$. In this case,
 there exists at most one $z\in L_G(v_{s-1})$ such that $z|f_i$ for some $i\in[t]$. Assuming by contradiction that there exist $z_1,z_2\in L_G(v_{s-1})$ with  $z_1\ne z_2$ such that  $z_1|f_i$ and  $z_2|f_j$ for some $i,j \in[t]$.  We can write $f_i$ and $f_j$ as  $f_i=(v_{s-1}z_1)^{\omega(v_{s-1}z_1)}$ and $f_j=(v_{s-1}z_2)^{\omega(v_{s-1}z_2)}$. Since $z_1\ne z_2$, $i\ne j$. Thus, 
$$\deg_{v_{s-1}}(gf) \geqslant\deg_{v_{s-1}}(f_if_j)=\omega(v_{s-1}z_1)+\omega(v_{s-1}z_2) =2\mu(v_{s-1}).$$
Since $v_{s-1}\nmid g$, 
$\deg_{v_{s-1}}(gf) = \deg_{v_{s-1}}(f)=\omega(v_{s-2}v_{s-1})+\mu(v_{s-1})-1=2\mu(v_{s-1})-1$. The second equality holds due to the expression of $f$.
This is a contradiction.

Since  $|L_G(v_{s-1})| \geqslant 2$, there is at least one $z\in L_G(v_{s-1})$ such that $z\nmid f_i$ for any $i\in [t]$.
Without loss of generality, we can assume that $v_s\nmid f_j$ for any $j\in [t]$. Thus, each $f_i\in \mathcal{G}(I(G''_\omega))$. 
Therefore, from  the expression of $f$, $$\deg_{v_s}(h)=\deg_{v_s}(gf)=\deg_{v_s}(g)+\deg_{v_s}(f)=\deg_{v_s}(g)+\omega(v_{s-1}v_s)-1.$$
We can write $h$ and $g$ as  $h = v_s^{d+\omega(v_{s-1}v_s)-1}h_2$ and $g = v_s^{d}g_1$, where $d=\deg_{v_s}(g)$, monomials $h_2$ and $g_1$ satisfy $v_s\nmid h_2g_1$.
 Since  $S(r,G_\omega)=S(r,G''_\omega)$, we  obtain   $f=f''v_s^{\omega(v_{s-1}v_s)-1}$,  where $f''=f(r,G''_\omega)$.
Therefore,  the expression $(\ref{EQ5})$ becomes 
$$g_1f'' = h_2f_1\cdots f_t\in I(G''_\omega)^t.$$
Note that $|S(r,G_\omega)|=|S(r,G''_\omega)|$. Then, by the induction hypothesis,  $g\in (A(G''_\omega)) + I(K_{U'',V''})\subset (A(G_\omega)) + I(K_{U,V})$, where $U''\sqcup V''=(U\sqcup V)\setminus \{v_s\}$.

\medskip
{\it Subcase $2.2$}: If $L_G(v_{s-1})=\{v_s\}$, then $\omega(v_{s-1}v_{s-2}) \leqslant \omega(zv_{s-2})$  for all  $z\in N_G(v_{s-2})$.
If  $L_G(v_{s-2})=\emptyset$, then, from  the condition (ii), we have that $\omega(v_{s-1}v_{s-2}) \leqslant \omega(zv_{s-2})$  for all  $z\in N_G(v_{s-2})$.
Now, we assume that $L_G(v_{s-2})\neq \emptyset$.  Using the condition (ii) again,   we only need to  show  that $\omega(v_{s-1}v_{s-2}) \leqslant \omega(zv_{s-2})$ 
for  all  $z\in L_G(v_{s-2})$. Suppose for contradiction that there is a $u \in L_G(v_{s-2})$ such that $\omega(v_{s-2}v_{s-1})> \omega(v_{s-2}u)$. Then, by the condition (ii), we have that $\omega(v_{s-2}z)>\omega(v_{s-2}u)$ for all $z\in N_G(v_{s-2})\setminus L_G(v_{s-2})$. This implies that 
$S(r,G_\omega)=S(r,(G\setminus u)_{\omega})$.  By  Definition \ref{notation},  $u\in A(G_\omega)$. Therefore,  $u\nmid g$, and  the pendant edge $v_{s-2}u$ satisfies  
the four conditions that  $\deg_G(u)=1$, $u\ne r$, $u\nmid g$  and $S(r,G_\omega)=S(r,(G\setminus u)_{\omega})$.
This contradicts the assumption.

Note that $L_G(v_{s-1})=\{v_s\}$ and that $\omega(v_{s-1}v_s)=\omega(v_{s-2}v_{s-1})$. Therefore, we can then deduce that $S(r,G_\omega) = S(r,G''_\omega)\cup \{v_{s-2}v_{s-1}\}$. Thus,
$$f = f''(v_{s-2}v_{s-1})^{\omega(v_{s-2}v_{s-1})}v_s^{\omega(v_{s-1}v_s)-1}.$$
We can write $g$ as  $g = v_s^pg_2$ for some monomial $g_2$,  where $p\geqslant 0$ and  $v_s\nmid g_2$. There are two subcases:

(1) If $p=0$, then $v_s\nmid g$, and it follows that 
$\deg_{v_s}(gf)=\deg_{v_s}(f)=\omega(v_{s-1}v_s)-1$. By the expression  $(\ref{EQ5})$,  $v_s\nmid f_i$ for any $i\in[t]$. This implies that  $f_i\in \mathcal{G}(I(G''_\omega))$ and that $v_s^{\omega(v_{s-1}v_s)-1}\mid h$. Let  $h = v_s^{\omega(v_{s-1}v_s)-1}h_3$. Then
$$gf = gf''(v_{s-2}v_{s-1})^{\omega(v_{s-2}v_{s-1})}v_s^{\omega(v_{s-1}v_s)-1}=v_s^{\omega(v_{s-1}v_s)-1} h_3f_1\cdots f_t.$$
Therefore, 
$$gf''\in (I(G''_\omega)^t\colon (v_{s-2}v_{s-1})^{\omega(v_{s-2}v_{s-1})}).$$
According to Lemma \ref{colon}, $(I(G''_\omega)^t \colon (v_{s-2}v_{s-1})^{\omega(v_{s-2}v_{s-1})}) =  I(G''_\omega)^{t-1}$. Thus, $g \in ( I(G''_\omega)^{t-1}:f'')$.
Since $S(r,G_\omega) = S(r,G''_\omega)\cup \{v_{s-2}v_{s-1}\}$, then  $t-1=|S(r,G''_\omega)|+1$. By the induction hypothesis, $( I(G''_\omega)^{t-1}:f'')\in (A(G''_\omega)) + I(K_{U'',V''})$. Therefore,
$g\in (A(G''_\omega)) + I(K_{U'',V''})\subseteq (A(G_\omega)) + I(K_{U,V})$. 

(2) If $p\geqslant 1$, then $v_s| g$. If $v_{s-1}|g$, then $g\in(v_{s-1}v_s)\subseteq I(K_{U,V})$. Otherwise, $v_s| f_j$ for some $j\in[t]$, since $g\in \mathcal{G}(I^t:f)$. Let $v_s|f_t$, then $f_t=(v_{s-1}v_s)^{\mu(v_{s-1})}$, since $v_s$ is a leaf of the path $\P$. By the expressions of $g$ and $f$, we can obtain that
$$gf = g_2v_s^{p-1}f''(v_{s-2}v_{s-1})^{\mu(v_{s-1})}v_s^{\mu(v_{s-1})} = hf_1\cdots f_{t-1}(v_{s-1}v_s)^{\mu(v_{s-1})}.$$
Thus 
$$g_2v_s^{p-1}v_{s-2}^{\mu(v_{s-1})}f'' = hf_1\cdots f_{t-1}.$$
Since $g=v_s^{p}g_2$ and  $v_{s-1}\nmid g$, we have that $v_{s-1}\nmid g_2$ and  $\deg_{v_{s-1}}(g_2v_s^{p-1}v_{s-2}^{\mu(v_{s-1})}f'')=\deg_{v_{s-1}}(f'')=\mu(v_{s-1})-1$.
Thus,  $v_s\nmid f_i$ for all $i\in[t-1]$, and  $v_s^{p-1}| h$ and $f_i\in \mathcal{G}(I(G''_\omega))$. Let $h = v_s^{p-1}h_4$, then
$$g_2v_{s-2}^{\mu(v_{s-1})}\in (I(G''_\omega)^{t-1}\colon f'').$$
 By the induction hypothesis,  $g_2v_{s-2}^{\mu(v_{s-1})}\in (A(G''_\omega)) + I(K_{U'',V''})\subseteq (A(G_\omega)) + I(K_{U,V})$.

If $v_{s-2}\in A(G_\omega)$, then, by Definition \ref{notation},  there is a simple path
$$v_{s-2} = z_1\to z_2\to\cdots\to z_{2\ell}$$
such that  $\omega(z_{2\ell-1}z_{2\ell}) < \mu(z_{2\ell})$. Therefore,  there is a simple path
$$v_s\to v_{s-1}\to v_{s-2} = z_1\to z_2\to\cdots\to z_{2\ell}$$
with $\omega(z_{2\ell-1}z_{2\ell}) < \mu(z_{2\ell})$. Thus,  $v_s\in A(G_\omega)$ by Definition \ref{notation}. Therefore,  $g\in (A(G_\omega))$, since $v_s|g$. This contradicts to $g\notin (A(G_\omega))$. Therefore,  $v_{s-2}\notin A(G_\omega)$. 
Note that $g_2\notin (A(G_\omega))$ since $g\notin (A(G_\omega))$. Thus, $g_2v_{s-2}^{\mu(v_{s-1})}\notin (A(G_\omega))$. 
Since $g_2v_{s-2}^{\mu(v_{s-1})}\in (A(G_\omega)) + I(K_{U,V})$, $g_2v_{s-2}^{\mu(v_{s-1})}\in I(K_{U,V})$.  Therefore, there exist $a\in U$ and $b\in V$ such that $ab\mid g_2v_{s-2}^{\mu(v_{s-1})}$. 
Since   $V(G)=U\sqcup V$, we can  assume that $v_s\in U$ by  symmetry. Then,  $v_{s-1}\in V$ and  $v_{s-2}\in U$ by the path $\P$.  Since $b\mid g_2v_{s-2}^{\mu(v_{s-1})}$ and   $b\ne v_{s-2}$, $b\mid g_2$. Therefore, $b|g$, which implies that $v_sb\mid g$. Therefore,  $g\in I(K_{U,V})$. We  have completed the proof.
\end{proof}

\begin{Lemma}\label{UV} $U\nsubseteq A(G_\omega)$ and $V\nsubseteq A(G_\omega)$.
\end{Lemma}
\begin{proof}  Let  $r\in U$, then $r\notin A(G_\omega)$, which implies  $U\nsubseteq A(G_\omega)$. Let $v\in N_G(r)$ such that $\omega(rv)=\mu(r)$. Then,   $v\notin A(G_\omega)$, and $N_G(r)\subseteq V$. The result follows.
\end{proof}

Now we are ready to prove the first main result of this section.

\begin{Theorem}\label{theorem-depth}
Let $G_\omega$ be an increasing weighted tree. Then, 
$$\depth(S/I(G_\omega)^t)=1 \text{ for all } t\geqslant s(G_\omega)+1.$$
In particular, $\dstab(I(G_\omega))\leqslant s(G_\omega)+1$.
\end{Theorem}
\begin{proof} According to Lemma \ref{max-ideal}, we have $\depth (S/I(G_\omega)^t)\geqslant 1$ for all $t\geqslant 1$. By Lemma \ref{depth-decreasing}, it suffices to show that $\depth (S/I(G_\omega)^{t_0})\leqslant 1$, where $t_0=s(G_\omega)+1$.

By Lemma \ref{sG-Equality}, there is a root $r$ of $G_\omega$ such that $s(G_\omega)=|S(r,G_\omega)|$. By Lemmas $\ref{L1}$ and $\ref{L2}$, we have 
$$(I(G_\omega)^{t_0} \colon f(r,G_\omega))= (A(G_\omega)) + I(K_{U,V})$$
where $(U,V)$ is a bipartition of $G$.
Let $U_1 = U\setminus A(G_\omega)$ and $V_1=V\setminus A(G_\omega)$. Then  $(A(G_\omega)) + I(K_{U,V})=(A(G_\omega)) + I(K_{U_1,V_1})$. By Lemma \ref{UV}, we have $U_1\ne \emptyset$ and $V_1\ne\emptyset$. Since $A(G_\omega) \sqcup U_1\sqcup V_1 = V(G)$, by Lemma \ref{sum1}(1), we have 
\begin{align*}
\depth (S/(I(G_\omega)^{t_0}\colon f(r,G_\omega)))&=\depth (S/((A(G_\omega))+I(K_{U_1,V_1})))\\
&=\depth (\KK[U_1\cup V_1]/I(K_{U_1,V_1}))=1.    
\end{align*}
By Lemma \ref{depth}, we obtain
$$\depth (S/I(G_\omega)^{t_0}) \leqslant \depth (S/(I(G_\omega)^{t_0}\colon f(r,G_\omega))) =1,$$
and the theorem follows.
\end{proof}

The following result characterizes $G_\omega$ such that $I(G_\omega)$ has a constant depth function.

\begin{Proposition}\label{constant-depth} Let $G_\omega$ be an increasing weighted tree. Then, $\depth (S/I(G_\omega)^t) = 1$ for all $t\geqslant 1$ if and only if $G_\omega$ is a strictly increasing weighted tree.
\end{Proposition}
\begin{proof} Suppose that $G_\omega$ is a strictly increasing weighted tree with a root $r$. Then, $s(G_\omega)=0$. According to Theorem \ref{theorem-depth},  $\depth (S/I(G_\omega)^t) = 1$ for all $t\geqslant 1$.

Now, suppose $\depth (S/I(G_\omega)^t) = 1$ for all $t\geqslant 1$. We will show that $G_\omega$ is a strictly increasing weighted tree. For every monomial $f\notin I(G_\omega)$,  observe that  $\sqrt{I(G_\omega)\colon f}$ is of the form $(W) + I(H)$, where $W\subseteq V(G)$ and $H$ is a subgraph of $G$ on the set $V(G)\setminus W$. Thus, $\sqrt{I(G_\omega)\colon f}$ is sequentially Cohen-Macaulay due to \cite[Theorem 1.2]{FV}. According to \cite[Proposition 2.23]{JS}, $I(G_\omega)$ is sequentially Cohen-Macaulay. Hence, by \cite[Theorem 4]{F}, we have
$$\depth (S/I(G_\omega)) = \min\{n-\height(P) \mid P\in \Ass(I(G_\omega))\}.$$
Together with \cite[Theorem 2.10]{LTZ2}, it implies that there exists a strong vertex cover $C$ such that $P=(C)$,  $|C|=n-1$ and $s(C)=0$.
Let $V(G) \setminus C = \{u\}$. For any simple path  in $G$ from a leaf $v_1$ to $u$
$$ v_1\to v_2\to \cdots \to v_{k-1}\to v_{k}=u,$$
with $k\geqslant 3$, we have that  $v_{k-1}\in N_G(u)$, $v_1,\ldots,v_{k-2}\notin N_G(u)$. 

Note that $\omega(v_{k-2}v_{k-1})< \omega(v_{k-1}u)$ as $C$ is a  strong vertex cover. On the other hand, since $s(C)=0$, $v_j$ is not special for every $j=2,\ldots,k-2$. Together with \cite[Lemma 1.9]{LTZ2}, it gives $\omega(v_{j-1}v_{j}) < \omega(v_{j}v_{j+1})$ for $j=2,\ldots,k-2$. Therefore, the weight $\omega$ on our path is strictly increasing, and therefore $(G_\omega,u)$ is a strictly increasing weighted tree.
\end{proof}

\medskip
We will conclude this section by presenting an example which demonstrates that if $G_\omega$ is a weighted tree but it is not an increasing weighted tree, then the stable value of the depth function of $I(G_\omega)$ may be any integer.

\begin{Example} {\em (\cite[Theorem 4.10]{LTZ})} For any  positive integer $d$, there is a weighted tree $G_\omega$ with $2d$ vertices such that $\dim (S/I(G_\omega))=d$ and $I(G_\omega)^t$ is Cohen-Macaulay for all $t\geqslant 1$. In particular,
$$\depth (S/I(G_\omega)^t) = d  \ \text{ for all } t\geqslant 1.$$
\end{Example}

\section{Regularity of powers of the edge ideal of weighted trees}
\label{sec:regularity}

In the previous section, we saw that if $I(G_\omega)$ is the edge ideal of a strictly increasing weighted tree $G_\omega$, then $\pd (S/I(G_\omega)) = n-1$. In this section, we explore  the homological aspects of such an ideal in more detail. We will assume that $(G_\omega,r)$ is a strictly increasing weighted tree with the vertex set  $V(G)=\{x_1,\ldots,x_n\}$. First, we find the minimal free resolution for  $I(G_\omega)$, and then we will compute the regularity of powers of this ideal.

We will start with the definition of the Taylor complex. Let $M = \{m_1,\ldots,m_k\}$ be a set of monomials of $S$,   and let $U$ be a subset of  $\{1,\ldots,k\}$, let
$$m_U= \lcm(m_i\mid i\in U)$$ be  the least common multiple of the set $\{m_i\mid i\in U\}$.

Let $a_U\in\NN^n$ be the exponent vector of  $m_U$, and let $S(-a_U)$ be the free module generated by the homogeneous element $e_U$ of degree $a_U$. Then, the Taylor complex $\mathbf{T} = \mathbf{T}(m_1,\ldots,m_k)$ is defined by
$$0 \longrightarrow T_k\longrightarrow T_{k-1}\longrightarrow\cdots\longrightarrow T_1\longrightarrow T_0$$
where
$$T_p = \bigoplus_{|U|=p}S(-a_U)$$
and the differential $d_p\colon T_p \to T_{p-1}$ is given by
$$d_p(e_U) = \sum_{i\in U} \sign(i,U)\frac{m_U}{m_{U\setminus i}}e_{U\setminus i}$$
where $\sign(i,U)$ is $(-1)^{j-1}$ if $i$ is the $j$-th element in the increasing ordering of $U$. It is well known that the Taylor complex $\mathbf T(m_1,\ldots,m_k)$ is a free resolution of $S/(M)$ (see \cite[Theorem 26.7]{P}) but it may not be  minimal (see \cite[Example 26.8]{P}).

Since $G$ is a tree, $|E(G)| = n-1$, therefore,  $\mathcal G(I(G_\omega)) = \{f_1,\ldots,f_{n-1}\}$. We will show that the Taylor complex $\mathbf{T} = \mathbf{T}(f_1,\ldots,f_{n-1})$ of $I(G_\omega)$ is a minimal free resolution of $S/I(G_\omega)$. Therefore, it provides an alternative  explanation for $\pd(S/I(G_\omega))=n-1$.

\begin{Lemma}\label{LCM-TWO} Let $(G_\omega,r)$ be a strictly increasing weighted tree. Then, for all subsets $W_1,W_2\subseteq \mathcal{G}(I(G_\omega))$ such that $W_1\subseteq W_2$ and $W_1\ne W_2$, we have 
 $\lcm(W_1) \ne \lcm(W_2)$.
\end{Lemma}
\begin{proof}   It suffices to consider the case $W_1=\{f_1,\ldots,f_k\}$ and $W_2=\{f_1,\ldots,f_k,f\}$, where $f_1,\ldots,f_k,f$ are distinct monomials in $\mathcal{G}(I(G_\omega))$.
Let  $f=(uv)^{\omega(uv)}$. If either $u$ or $v$ is not in $\supp(W_1)$, then the statement holds. Therefore, we assume that $u,v\in \supp(W_1)$. There is a simple path from the root $r$ of the form
$$r = u_1\to u_2\to\cdots\to u_{k-1}\to u_k = u \to v.$$
Let $G'$ and $G''$ be the two disjoint subtrees obtained by deleting the edge $uv$ from $G$, where $u\in V(G')$ and $v\in V(G'')$.
Then, $(G'_\omega,r)$ and $(G''_\omega,v)$ are two strictly increasing weighted trees. Let 
$$W' = W_1\cap \mathcal G(I(G'_\omega)) \text{ and } W'' = W_1\cap \mathcal G(I(G''_\omega)).$$
Then, $\lcm(W_1) = \lcm(W') \lcm(W'')$ since $\supp(W')\cap \supp(W'')=\emptyset$. Because $v\notin \supp(W')$, we have
$$\deg_v(\lcm(W_1))=\deg_v(\lcm(W'')).$$
Note that $\omega(uv) > \max\{\omega(vw)\mid w\in N_{G''}(v)\}$ since $(G_\omega,r)$ is a strictly increasing weighted tree. It follows that $\deg_v(\lcm(W_2)) \geqslant\deg_v(f)= \omega(uv) > \deg_v(\lcm(W_1))$, so $\lcm(W_1)\ne \lcm(W_2)$, as required.
\end{proof}

Now we are ready to prove the first main result of this section.

\begin{Theorem} \label{minimal} Let $G_\omega$ be a strictly increasing weighted tree. Then, the Taylor complex $\mathbf T$ of $I(G_\omega)$ is the minimal free resolution of $S/I(G_\omega)$.
\end{Theorem}
\begin{proof} According to \cite[Theorem 26.7]{P},  $\mathbf{T}$ is a free resolution of $S/I(G_\omega)$. It remains to prove that the differential $d_p \colon T_p \to T_{p-1}$ satisfies $d_p(T_p) \subseteq \mm T_{p-1}$ for all $p\in[n-1]$. For each $U\subseteq \{1,\ldots,n-1\}$ with $|U| = p$, we have
$$d_p(e_U) = \sum_{i\in U} \sign(i,U)\frac{m_U}{m_{U\setminus i}}e_{U\setminus i}.$$
By Lemma \ref{LCM-TWO},  $m_U \ne m_{U\setminus i}$ for all $i\in U$, so $\deg(m_U/m_{U\setminus i})\geqslant 1$. Therefore,  $d_p(e_U) \in\mm T_{p-1}$, and the result follows.
\end{proof}

\begin{Corollary}\label{totalBetti} Let $G_\omega$ be a strictly increasing weighted tree. Then, 
$$\beta_i(S/I(G_\omega)) = \binom{n-1}{i} \ \text{ for every }i=0,\ldots,n-1.$$
\end{Corollary}
\begin{proof} Let $I = I(G_\omega)$ and $\mathcal G(I)=\{f_1,\ldots,f_{n-1}\}$. Since the Taylor complex $\mathbf T = \mathbf T(f_1,\ldots,f_{n-1})$ is a minimal free resolution of $S/I$, thus, for any $i=0,\ldots,n-1$,
$$\beta_i(S/I) = \rank (T_i) = |\{W\mid  W\subseteq \mathcal G(I) \text { and } |W| = i\}|=\binom{n-1}{i},$$
as required.
\end{proof}

\begin{Lemma}\label{LCM} Let $G_\omega$ be an increasing weighted tree. Then, 
$$\deg(\lcm(\mathcal{G}(I(G_\omega)))) = d+\sum_{e\in E(G)} \omega(e),$$
where  $d = \max\{\omega(e)\mid e\in E(G_\omega)\}$.
\end{Lemma}
\begin{proof}
Let $I=I(G_\omega)$ and $n=|V(G)|$. We will prove the statement by induction on $n$. The case $n=2$ is trivial. Now, assume that $n\geqslant 3$. If $G_\omega$ is a star graph with a center $r$, then $I = ((rv)^{\omega(rv)}\mid v\in N_G(r)\}$.  Therefore, $\lcm(\mathcal{G}(I)) = r^d\prod\limits_{v\in N_G(r)} v^{\omega(rv)}$. Therefore,
$\deg(\lcm(\mathcal{G}(I))) = d + \sum_{v\in N_G(r)} \omega(rv)=d + \sum_{e\in E(G)} \omega(e)$.

In the following, we assume that $G$ is not a star graph. Let $r$ be a root of $G_\omega$. By Lemma \ref{path}, there is a longest path from $r$ in the form
$$r = v_0 \to v_1\to v_2\cdots \to v_{k-1}\to v_k,$$
where $k\geqslant 2$, $v_k$ is a leaf and $\omega(v_kv_{k-1}) \leqslant \omega(v_{k-1}v)$ for all $v\in N_G(v_{k-1})$. 
Let $I'=I(G'_\omega)$, where $G'=G\setminus v_k$. Then, 
$$\lcm(\mathcal{G}(I)) = \lcm(\mathcal{G}(I')\cup\{(v_{k-1}v_k)^{\omega(v_{k-1}v_k)}\}).$$ 
 Since $\omega(v_kv_{k-1}) \leqslant \omega(v_{k-1}v)$ for all $v\in N_G(v_{k-1})$, 
$$\lcm(\mathcal{G}(I')\cup\{(v_{k-1}v_k)^{\omega(v_{k-1}v_k)}\}) = \lcm(\mathcal{G}(I'))v_k^{\omega(v_{k-1}v_k)}.$$ 
By the induction hypothesis, we can get
$$\deg(\lcm(\mathcal{G}(I')))=d+\sum_{e\in E(G')}\omega(e).$$
Therefore, 
$$\deg(\lcm(\mathcal{G}(I))) = \deg(\lcm(\mathcal{G}(I')))+\omega(v_{k-1}v_k) =d+\sum_{e\in E(G)}\omega(e),$$
as required.
\end{proof}

The following result provides  the formula for $\reg I(G_\omega)$.

\begin{Proposition} \label{regI} Let $G_\omega$ be a strictly increasing weighted tree. Then, 
$$\reg(I(G_\omega)) = d+\sum_{e\in E(G)}(\omega(e)-1)+1,$$ 
where $d=\max\{\omega(e)\mid e\in E(G)\}$.
\end{Proposition}

\begin{proof} Let $I = I(G_\omega)$  and $n=|V(G)|$.  For every subset $W\subseteq \mathcal G(I)$, by Lemma \ref{LCM-TWO}, we have
$$\deg(\lcm(\mathcal{G}(I))) \geqslant \deg(\lcm(W)) + (n-1 - |W|).$$
According to  Theorem \ref{minimal},  $\reg (I) = \deg(\lcm(\mathcal{G}(I)))-(n-1)+1$. Combining with the fact that $|E(G)|=n-1$ and Lemma \ref{LCM}, we obtain
$$\reg (I) = d + \sum_{e\in E(G)}\omega(e) -(n-1)+1 = d + \sum_{e\in E(G)}(\omega(e)-1)+1,$$
as required.\end{proof}

The next our goal is to compute $\reg I(G_\omega)^t$ for which we start with the basic case.  
\begin{Lemma} {\em (\cite[Theorem 3.3(2)]{ZDCL})} \label{star} Let $G_\omega$ be a weighted star graph. Then,
$$\reg(I(G_\omega)^t)=2d(t-1) + \reg (I(G_\omega)) \text{ for all } t\geqslant 1,$$
where $d=\max\{\omega(e)\mid e\in E(G)\}$.
\end{Lemma}

We are now ready to prove second main result of this section.

\begin{Theorem}\label{reg-power} Let $G_\omega$ be a strictly increasing weighted tree. Then
$$\reg(I(G_\omega)^t)=2d(t-1) + \reg (I(G_\omega)) \text{ for all } t\geqslant 1,$$
where $d=\max\{\omega(e)\mid e\in E(G)\}$.
\end{Theorem}
\begin{proof} Let $I=I(G_\omega)$. We will prove the statement  by induction on $t$ and $n = |V(G)|$. The case  $t=1$  is trivial. Now, assume that $t\ge 2$. If $G_\omega$ is a star graph, then the result follows from Lemma \ref{star}. Now, assume that $G_\omega$ is not a star graph with a root  $r$. By Lemma \ref{path}, there is a longest simple path in $G_\omega$ from $r$ in the form
$$r = v_0 \to v_1\to v_2\to \cdots \to v_{k-1}\to v_k$$
such that $k\geqslant 2$, $v_k$ is a leaf and $\omega(v_{k-1}v_k) \leqslant \omega(v_{k-1}v)$ for every $v\in N_G(v_{k-1})$. Let $m = \omega(v_{k-1}v_k)$. Then  $m < d$, since $(G_\omega,r)$ is a strictly increasing weighted tree.

Let  $G' = G\setminus \{v_k\}$ and $G'' = G\setminus (L_G(v_{k-1})\cup\{v_{k-1}\})$. Then  
\[
((I^t,(v_{k-1}v_k)^{m}):v_k^{m})=(v_{k-1}^{m})+I(G'_\omega)^t=(v_{k-1}^{m})+I(G''_\omega)^t,
\]
and
$$((I^t,(v_{k-1}v_k)^{m}),v_k^{m})=I^t+(v_{k}^{m})=I(G'_\omega)^t+(v_{k}^{m}).
$$
Let $d'=\max\{\omega(e)\mid e\in E(G'_\omega)\}$ and $d''=\max\{\omega(e)|e\in E(G''_\omega)\}$. Then $d'=d$, and $d''\leqslant d$. Note that  $(I^t:(v_{k-1}v_k)^{m})=I^{t-1}$ by Lemma \ref{colon}.

Using  Lemma \ref{sum1}(2), Proposition \ref{regI}, as well as  the induction on $n$ and $t$, we can conclude that 
\begin{align*}
	\reg((I^t:(v_{k-1}v_k)^{m}))&=\reg(I^{t-1})=2d(t-2)+d+\sum\limits_{e\in E(G)}(\omega(e)-1)+1\\
	&< 2d(t-1)+d+\sum\limits_{e\in E(G)}(\omega(e)-1)+1,\\
	\reg(((I^t,(v_{k-1}v_k)^{m}):v_k^{m}))&=\reg((v_{k-1}^{m})+I(G''_\omega)^{t})\\
&=2d''(t-1)+d''+\sum\limits_{e\in E(G'')}(\omega(e)-1)+1+(m-1)\\
	&< 2d(t-1)+d+\sum\limits_{e\in E(G)}(\omega(e)-1)+1,\\
\reg(((I^t,(v_{k-1}v_k)^{m}),v_k^{m}))&=\reg((v_{k}^{m})+I(G'_\omega)^{t})\\
&=2d'(t-1)+d'+\sum\limits_{e\in E(G')}(\omega(e)-1)+1+(m-1)\\
		&=2d(t-1)+d+\sum\limits_{e\in E(G)}(\omega(e)-1)+1.\\
\end{align*}
Applying Lemma \ref{exact} to the following two short exact sequences
\begin{gather*}
	\begin{matrix}
		0 & \!\!\longrightarrow\!\!\!  & \frac{S}{I^t:(v_{k-1}v_k)^{m}}(-2m) &\!\!\stackrel{\cdot (v_{k-1}v_k)^{m}} \longrightarrow\!\! & \frac{S}{I^t} & \!\!\longrightarrow\!\!  & \frac{S}{(I^t,(v_{k-1}v_k)^{m})} & \!\!\!\longrightarrow\!\!  & 0,\\
		0 & \!\!\longrightarrow\!\!\!  & \frac{S}{(I^t,(v_{k-1}v_k)^{m}):v_k^{m}}(-m) & \!\!\stackrel{\cdot v_k^{m}}\longrightarrow\!\! & \frac{S}{(I^t,(v_{k-1}v_k)^{m})} & \!\!\longrightarrow\!\!  & \frac{S}{((I^t,(v_{k-1}v_k)^{m}),v_k^{m})} & \!\!\!\longrightarrow\!\!  & 0,\\
	\end{matrix}
\end{gather*}
we have $\reg(S/I(G_\omega)^t)=2d(t-1)+\reg(S/I(G_\omega))$. The proof is complete. 
\end{proof}

\medskip
The following example shows that Theorem \ref{reg-power} does not hold in the case $G_\omega$ is an increasing but not strictly increasing weighted tree.

\begin{Example}\label{ExampleReg} Let $G$ be a path of length $3$ with the edge set $E=\{x_1x_2,x_2x_3,x_3x_4\}$. Consider the weight function $\omega$ such that
$$\omega(x_1x_2)=6,\omega(x_2x_3)=\omega(x_3x_4)=5.$$
Then, $I(G_\omega) = ((x_1x_2)^6,(x_2x_3)^5,(x_3x_4)^5)$. Using Macaulay2, we found that
$$\reg (I(G_\omega) )= 16 \text{ and } \reg (I(G_\omega)^2 )= 30.$$
Since $d=6$, we have $\reg(I(G_\omega)^t) >2 d(t-1)+\reg(I(G_\omega))$ for $t=2$.
\end{Example}

\medskip

\hspace{-6mm} {\bf Acknowledgments}
		
		\vspace{3mm}
		\hspace{-6mm}  The fourth author is  supported by the Natural Science Foundation of Jiangsu Province (No. BK20221353) and the National Natural Science Foundation of China (12471246). The first and third authors are partially supported by Vietnam National Foundation for Science and Technology Development (Grant \#101.04-2024.07). The main part of this
work was done during the third author's visit to Soochow University in Suzhou, China.  He would like to express his gratitude to Soochow University for its warm hospitality.

		\medskip
		\hspace{-6mm} {\bf Data availability statement}
		
		\vspace{3mm}
		\hspace{-6mm}  The data used to support the findings of this study are included within the article.
		
		\medskip
		\hspace{-6mm} {\bf Conflict of interest}
		
		\vspace{3mm}
		\hspace{-6mm}  The authors declare that they have no competing interests.

\end{document}